\documentclass[12pt,leqno]{article} %11pt,
\usepackage{amsmath,amstext, amsthm, amssymb,graphicx,epstopdf}
\usepackage{soul,color}
\usepackage{hyperref}
  \numberwithin{equation}{section}
  \def\N{\mathbb{ N}}
\def\R{\mathbb{ R}}  
  
\usepackage{epsfig}
\newtheorem{thm}{Theorem}[section]
\newtheorem{lem}[thm]{Lemma}

\newtheorem{conj}[thm]{Conjecture}
\newtheorem{prop}[thm]{Proposition}
\newcommand{\be}{\begin{equation}}
\newcommand{\ee}{\end{equation}}

\newcommand{\ba}{\begin{array}}
\newcommand{\ea}{\end{array}}

\renewcommand{\d}{\delta}
\newcommand{\bg}{\begin{gathered}}
\newcommand{\eg}{\end{gathered}}

\renewcommand{\a}{\alpha}
\renewcommand{\b}{\beta}
\newcommand{\ep}{\varepsilon}
\renewcommand{\th}{\theta}

\newtheorem{rem}[thm]{Remark}
\newcommand{\bea}{\begin{eqnarray}}
\newcommand{\eea}{\end{eqnarray}}

\def\vp{\varphi}

\begin{document}

\title{Orthogonal polynomials with periodic recurrence coefficients}
\author{Dan Dai, Mourad E. H. Ismail, and Xiang-Sheng Wang}
%\date{}

\maketitle

\begin{abstract}
  In this paper, we study a class of orthogonal polynomials defined by a three-term recurrence relation with periodic coefficients. We derive explicit formulas for the generating function, the associated continued fraction, the orthogonality measure of these polynomials, as well as the spectral measure for the associated doubly infinite tridiagonal Jacobi matrix. Notably, while the orthogonality measure may include discrete mass points, the spectral measure(s) of the doubly infinite Jacobi matrix are  absolutely continuous. Additionally, we uncover an intrinsic connection between these new orthogonal polynomials and Chebyshev polynomials through a nonlinear transformation of the polynomial variables.
 \end{abstract}



\section{Introduction}

The Chebyshev polynomials of the first and second kind are prototype of orthogonal polynomials on a compact interval. They satisfy differential and difference equations, have raising and lowering operators and explicit representations. They are the model for the Szeg\H{o} class of polynomials orthogonal with respect to  an  absolutely continuous measure $\mu$ on $[-1,1]$, where
$\int_0^\pi \mu'(\cos \theta) d\theta$ is finite. The purpose of this work is to develop a model example for polynomials that are orthogonal on several disjoint intervals. We expect this to contribute to a theory of orthogonal polynomials on multiple intervals that parallels Szeg\H{o}'s theory.

In the 1980s, Barry Simon and his research team were interested in the spectral theory of the discrete Schr\"odinger operator  represented by the Jacobi matrix $T$ whose elements are defined by
\begin{equation}
t_{j,k} = \delta_{j, k+1} +  \delta_{j, k-1} + a \cos (2\pi k \alpha+ \beta)  \delta_{j, k},
\end{equation}
where $\alpha$ is irrational.  Avron and Simon \cite{Avr:Sim1,Avr:Sim2} proved that the spectrum is a Cantor set.  The potential $a \cos (2\pi k \alpha+ \beta)$ is an almost periodic potential  and the corresponding operator is called an almost Mathieu operator; see Avron et al. \cite{Avr:Mou:Sim}.  This problem originated
with the study of imperfect crystals.

Given $a\in\R$ and $q=e^{2\pi i/N}$ with $N\in\N$, we consider the orthogonal polynomials determined by the three-term recurrence relation
\begin{equation}\label{Pn}
  2xP_n(x)=P_{n+1}(x)+a(q^n+q^{-n})P_n(x)+P_{n-1}(x),~~n\ge1,
\end{equation}
with initial conditions $P_0(x)=1$ and $P_1(x)=2x-2a$. For convenience, we also set $P_{-1}(x)=0$, ensuring that the above recurrence relation remains valid for $n=0$.
Consider $P_n(x)=P_n(x;a)$ as a function of both $x$ and $a$, we obtain by symmetry that $P_n(x;-a)=(-1)^nP_n(-x;a)$.
Moreover, the special case $a=0$ corresponds to the Chebyshev polynomials. Hence, throughout this paper, we shall assume without loss of generality $a>0$.
Even though we are mainly interested in deriving the orthogonality measure for $P_n(x)$ and finding its relation with Chebyshev polynomials, it is useful to introduce the corresponding numerator polynomials $P_n^*(x)$, which satisfy the same three-term recurrence relation as $P_n(x)$ but with different initial conditions. More specifically, they are defined by
\begin{equation}\label{Pn*}
  2xP_n^*(x)=P_{n+1}^*(x)+a(q^n+q^{-n})P_n^*(x)+P_{n-1}^*(x),~~n\ge1,
\end{equation}
with initial conditions $P_0^*(x)=0$ and $P_1^*(x)=2$. It is worth to mention that $P_n^*(x)$ is a polynomial of degree $n-1$, and any solution to the three-term recurrence relation \eqref{Pn} can be represented as a linear combination of $P_n(x)$ and $P_n^*(x)$. For instance, if we define $R_n(x):=P_{N+n}(x)$, the periodicity $q^{n+N}=q^n$ implies that $R_n(x)$ satisfies the same three-term recurrence relation as that for $P_n(x)$ and $P_n^*(x)$. Since $R_0(x)=P_N(x)$ and $R_1(x)=P_{N+1}(x)=(2x-2a)P_N(x)-P_{N-1}(x)$.
It then follows from linear dependence that $P_{N+n}(x)=R_n(x)=P_N(x)P_n(x)-P_{N-1}(x)P_n^*(x)/2$.
In particular, by choosing $n=N-1$, we obtain
\begin{equation}\label{Q}
  P_{2N-1}(x)=P_{N-1}(x)[P_N(x)-P_{N-1}^*(x)/2].
\end{equation}

One can think of the polynomials studied here as a generalization  of Chebyshev polynomials to several intervals. There are other known versions of such generalizations; see for example \cite{Che:Gri:Ism} and \cite{Ism}. Additionally, the polynomials $P_n(x)$ in \eqref{Pn} are generated from three-term recurrence relations. Our work connects to broader investigations of orthogonal polynomials with periodic or asymptotically periodic recurrence coefficients, along with the spectral properties of their associated Jacobi matrices. For related studies on such polynomial systems and their spectral analysis, we refer to \cite{Alm, Bou:Goo, Geronimo86JAT}.

The rest of the paper is organized as follows. In \S \ref{Sec:2}, we derive a surprisingly simple looking generating function for the polynomials $\{P_n(x)\}$. The generating function is repeatedly used in the later sections to study the polynomial system $\{P_n(x)\}$.  In \S \ref{Sec:3}, it is applied  to express $P_n(x)$ and $P_n^*(x)$ in terms of Chebyshev polynomials, where the generating function for the numerator polynomials is also given. In \S \ref{Sec:4}-\ref{Sec:6} we identify the continued fraction and find the orthogonality measure of the polynomials $\{P_n(x)\}$. In \S \ref{Sec:7}, we determine the spectral measure of  the associated doubly infinite Jacobi matrix  by applying techniques from \cite{Berbook} and \cite{Mas:Rep}.  A few concrete examples are presented in \S \ref{Sec:8}.

\section{Generating function} \label{Sec:2}
Fix $x$, the generating function
\begin{equation}\label{P}
  P(t):=\sum_{n=0}^\infty P_n(x)t^n
\end{equation}
satisfies the equation
\begin{equation}\label{P-eq}
  Q(t)P(t)+at[P(tq)+P(t/q)]=1,
\end{equation}
where
\begin{equation} \label{Q-def}
Q(t)=t^2-2xt+1.
\end{equation}
Replacing $t$ with $tq,\ldots,tq^{N-1}$ in the above equation, we obtain a linear system for $P(t),P(tq),\ldots,P(tq^{N-1})$. Solving this linear system gives an explicit expression of $P(t)$ as a rational function.
More specifically,
\begin{equation}\label{P-F}
  P(t)=\frac{F(t)}{\det[M(t)]},
\end{equation}
where the numerator is a polynomial in $t$ with degree $2N-2$ whose coefficients are polynomials of $x$, and the denominator is the determinant of the following coefficient matrix
\begin{align} \label{M-def}
  M(t)=\begin{pmatrix}
    Q(t)&at&&&&&at\\
    atq&Q(tq)&atq\\
    %&atq^2&Q(tq^2)&atq^2\\
    &\cdots&\cdots&\cdots\\
    &&\ddots&\ddots&\ddots\\
    &&&\cdots&\cdots&\cdots\\
    &&&&atq^{N-2}&Q(tq^{N-2})&atq^{N-2}\\
    atq^{N-1}&&&&&atq^{N-1}&Q(tq^{N-1})
  \end{pmatrix}.
\end{align}
The above matrix $M(t)$ is nearly a tridiagonal matrix, except that the $(1,N)$ and $(N,1)$ entries are given by $at$ and $atq^{N-1}$, respectively.

Then, we have the following theorem for the explicit expression of the generating function.

\begin{thm}\label{thm-P}
  The generating function for the orthogonal polynomials $P_n(x)$ is given by
  \begin{align} \label{generating-fun}
    \sum_{n=0}^\infty P_n(x)t^n=\frac{1}{t^{2N}-2g_N(x)t^N+1}\left[\sum_{k=0}^{2N-2}P_k(x)t^k-2g_N(x)\sum_{k=0}^{N-2}P_k(x)t^{k+N}\right],
  \end{align}
  where
  \begin{equation}\label{gN}
    g_N(x)=\frac{P_{2N-1}(x)}{2P_{N-1}(x)}=\frac{P_N(x)-P_{N-1}^*(x)/2}{2}.
  \end{equation}
\end{thm}
\begin{proof}
From the definitions of $Q(t)$ and $M(t)$ in \eqref{Q-def} and \eqref{M-def}, it is obvious that $\det[M(t)] $ is a monic polynomial of degree $2N$ in $t$, that is
\begin{equation}
\det[M(t)] = t^{2N} + \sum_{k=0}^{2N-1} m_k(x) t^{k}.
\end{equation}
Since $\det[M(t)] $ remains unchanged when we replace $t$ with $tq^k$, all the coefficients $m_k(x)$ in the above formula vanish, except for $m_0(x)$ and $m_N(x)$. Given that $M(0)$ is an identity matrix, we have $m_0(x) = \det[M(0)] = 1$. Therefore, we obtain
\begin{equation}\label{det}
  \det[M(t)]=t^{2N}-2g_N(x)t^N+1,
\end{equation}
where $g_N(x)$ is a polynomial of $x$ to be determined.

As $F(t)$ is a polynomial in $t$ with degree $2N-2$, let us rewrite it as
\begin{equation}\label{F}
  F(t)=\sum_{k=0}^{2N-2}F_k(x)t^k.
\end{equation}
With the expressions of $P(t)$ in \eqref{P} and $\det[M(t)]$ in \eqref{det}, we get
\begin{equation}
 P(t) \det[M(t)]
  =\sum_{k=0}^\infty P_k(x)t^k-2g_N(x)\sum_{k=N}^\infty P_{k-N}(x)t^k+\sum_{k=2N}^\infty P_{k-2N}(x)t^k.
\end{equation}
Since $F(t)=P(t) \det[M(t)] $, we obtain by comparing the like terms
\begin{align}\label{Fk}
  F_k(x)=
  \begin{cases}
  P_k(x),&~~k=0,\ldots,N-1,\\
  P_k(x)-2g_N(x)P_{k-N}(x),&~~k=N,\ldots,2N-2,
  \end{cases}
\end{align}
and
\begin{equation}
  P_k(x)-2g_N(x)P_{k-N}(x)+P_{k-2N}(x)=0,~~k\ge2N-1.
\end{equation}
Note that $P_{-1}(x)=0$. By setting $k=2N-1$ in the above formula, we obtain $P_{2N-1}(x)=2g_N(x)P_{N-1}(x)$. Together with \eqref{Q}, this yields \eqref{gN}. The expression in \eqref{generating-fun} also follows from combining the above three formulas. This completes the proof.
\end{proof}

\begin{rem}
 The function $g(z)$ defined in \eqref{gN} is related to the $N$-step transfer matrix $T_N$ in the spectral for Jacobi matrices. More precisely, from \cite[Eq. (10.60)]{Lukbook} or \cite[Eq. (5.4.3)]{Simonbook}, the $N$-step transfer matrix $T_N$ associated with the recurrence relation \eqref{Pn} is given by
 \begin{equation}
 T_N = \begin{pmatrix}
 P_N(z) & - P_N^*(z) \\
 \frac{1}{2} P_{N-1}(z) & -\frac{1}{2} P_{N-1}^*(z)
 \end{pmatrix}.
 \end{equation}
 Then, it is clear that $g_N(z) = \frac{1}{2} \textnormal{Tr} \, T_N$, where $\textnormal{Tr} \, T_N$ is called the  discriminant in \cite[Eq. (5.4.5)]{Simonbook}.
\end{rem}

\begin{rem}
The above theorem holds for all $N\in \mathbb{N}$. In particular, when $N=1$, we have from \eqref{Pn} that
\begin{equation}
2(x -a) P_n(x) = P_{n+1}(x) + P_{n-1} (x), \qquad n \geq 1,
\end{equation}
with $P_0(x) = 1$ and $P_1(x) = 2x - 2a$. This implies that $P_n(x)$ are indeed the Chebyshev polynomials of the second kind: $P_n(x) = U_n(x-a)$. Moreover, from  \eqref{gN}, we have $g_1(x) = x-a$. Then, when $N=1$, we get
\begin{equation} \label{Un-gene}
\sum_{n=0}^\infty P_n(x)t^n = \sum_{n=0}^\infty U_n(x-a)t^n = \frac{1}{t^2 -2(x-a)t +1},
\end{equation}
which is the generating function for $U_n$; see \cite[Eq. (18.12.10)]{NISTbook}.
\end{rem}

\section{Relation to Chebyshev polynomials} \label{Sec:3}
The polynomials $P_n(x)$ generated by \eqref{Pn} are related to the Chebyshev polynomials of the second kind $U_n(x)$, as described in the theorem below.
\begin{thm}\label{thm-k+jN}
Let $g_N=g_N(x)$ be given as in \eqref{gN}.
For any $k\ge0$ and $j\ge1$, we have
\begin{align}\label{k+jN}
  P_{k+jN}(x)=P_{k+N}(x)U_{j-1}(g_N)-P_k(x)U_{j-2}(g_N).
\end{align}
If $\th=\arccos g_N$, then
\begin{align}
  \frac{P_{k+jN}(x)}{P_k(x)}=\frac{\rho\sin(j\th+\vp)}{\sin\th},
\end{align}
where $\rho>0$ and $\vp\in[0,2\pi)$ are independent of $j=1,2,\ldots,$ and satisfy
\begin{align}
 % \rho^2=&[P_{k+N}(x)/P_k(x)]^2-2\cos\th[P_{k+N}(x)/P_k(x)]+1,\\
  \rho\cos\vp=&P_{k+N}(x)/P_k(x)-\cos\th,\\
  \rho\sin\vp=&\sin\th.
\end{align}
\end{thm}
\begin{proof}
From \eqref{det} and \eqref{Un-gene}, we have
\begin{equation}\label{M-U}
  \frac{1}{\det(M)}=\frac{1}{t^{2N}-2g_N(x)t^N+1}=\sum_{j=0}^\infty U_j(g_N)t^{jN},
\end{equation}
where $U_j(g_N)$ is the Chebyshev polynomial of second kind with $g_N=g_N(x)$ as the variable. It then follows from \eqref{P-F}, \eqref{F} and \eqref{Fk} that
\begin{align*}
  P(t)=\sum_{k=0}^{N-1}\sum_{j=0}^\infty P_k(x)U_j(g_N)t^{k+jN}
  +\sum_{k=0}^{N-2}\sum_{j=1}^\infty [P_{k+N}(x)-2g_NP_k(x)]U_{j-1}(g_N)t^{k+jN}.
\end{align*}
This together with \eqref{P} implies that
\begin{align*}
  P_{k+jN}(x)=P_k(x)U_j(g_N)+[P_{k+N}(x)-2g_NP_k(x)]U_{j-1}(g_N),
\end{align*}
for $k=0,\ldots,N-1$ and $j=1,2,\ldots,\infty$.
  In view of the recurrence relation $U_j(g_N)-2g_NU_{j-1}(g_N)=-U_{j-2}(g_N)$, the above equation gives \eqref{k+jN}.

If $\th=\arccos g_N$, then we have $U_n(g_N)=\sin[(n+1)\th]/\sin\th$; see \cite[Eq. (18.5.2)]{NISTbook}. The equation \eqref{k+jN} can be rewritten as
  \begin{align*}
    \frac{P_{k+jN}(x)\sin\th}{P_k(x)}=&\frac{P_{k+N}(x)}{P_k(x)}\sin(j\th)-\sin(j\th-\th)
    \\=&\left[\frac{P_{k+N}(x)}{P_k(x)}-\cos\th\right]\sin(j\th)+\sin\th\cos(j\th)
    \\=&\rho\sin(j\th+\vp),
  \end{align*}
  which completes the proof.
\end{proof}

The numerator polynomials $P_n^*(x)$ also have a similar relationship with $U_n(x)$.

\begin{thm}\label{thm-k+jN*}
  Let $g_N=g_N(x)$ be given as in \eqref{gN}.
For any $k\ge0$ and $j\ge1$, we have
\begin{align}\label{k+jN*}
  P_{k+jN}^*(x)=P_{k+N}^*(x)U_{j-1}(g_N)-P_k^*(x)U_{j-2}(g_N).
\end{align}
\end{thm}
\begin{proof}
Consider the generating function
\begin{equation}
  P^*(t):=\sum_{n=0}^\infty P_n^*(x)t^n,
\end{equation}
which satisfies the equation
\begin{equation}
  Q(t)P^*(t)+at[P^*(tq)+P^*(t/q)]=2t.
\end{equation}
Comparing the above formula with \eqref{P-eq}, the only difference is that the quantity $1$ on the right-hand side is replaced by $2t$. This change arises because we need to set $P_{-1}^*(x) = -2$ to ensure that \eqref{Pn*} holds when $n = 0$.

Next, a similar argument as in the proof of Theorem \ref{thm-P} shows that
\begin{align}
  P^*(t)=\frac{F^*(t)}{\det(M)}=\frac{1}{t^{2N}-2g_N(x)t^N+1}\sum_{k=0}^{2N-1}F^*_k(x)t^k,
\end{align}
where
\begin{equation}
  F^*_k(x)=
  \begin{cases}
  P_k^*(x),&~~k=0,\ldots,N-1,\\
  P_k^*(x)-2g_NP_{k-N}^*(x),&~~k=N,\ldots,2N-1.
  \end{cases}
\end{equation}
From \eqref{M-U}, we have
\begin{align}
  P_{k+jN}^*(x)=P_k^*(x)U_j(g_N)+[P_{k+N}^*(x)-2g_NP_k^*(x)]U_{j-1}(g_N),
\end{align}
for $k=0,\ldots,N-1$ and $j=1,2,\ldots,\infty$.
On account of the recurrence relation $U_j(g_N)-2g_NU_{j-1}(g_N)=-U_{j-2}(g_N)$, we obtain \eqref{k+jN*}.
\end{proof}

\section{The zeros of $P_N(x)$, $P_{N-1}(x)$, and $P_N^*(x)$} \label{Sec:4}

It is well-known that the zeros of orthogonal polynomials are simple and real. Moreover, they satisfy the following interlacing properties.
\begin{lem}\label{lem-yk}
  Let $x_1<x_2<\cdots<x_N$ be the zeros of $P_N(x)$.
  For any $k=1,\dots,N$, we have
  \begin{equation}\label{xk}
    (-1)^{N-k}P_N'(x_k)>0,~~(-1)^{N-k}P_{N-1}(x_k)>0.
  \end{equation}
  There exists a zero of $P_{N-1}(x)$, denoted by $y_k$, in each interval $(x_k,x_{k+1})$ with $k=1,\ldots,N-1$.
  Moreover,
  \begin{align}\label{yk-0}
    (-1)^{N-k}P_{N-1}'(y_k)<0,~~(-1)^{N-k}P_N(y_k)>0.
  \end{align}
\end{lem}
\begin{proof}
The results can be proved by using the Christoffel-Darboux formula for orthogonal polynomials; see \cite[Theorem 2.2.3]{Ismbook}.
\end{proof}
\begin{lem}\label{lem-ykzk}
  Let $x_1<x_2<\cdots<x_N$ be the zeros of $P_N(x)$.
  We have for each $k=1,\ldots,N,$
  \begin{equation}\label{xk*}
    P_{N-1}(x_k)P_N^*(x_k)=2,~~(-1)^{N-k}P_N^*(x_k)>0.
  \end{equation}
  There exists a zero of $P_N^*(x)$, denoted by $z_k$, in each interval $(x_k,x_{k+1})$ with $k=1,\ldots,N-1$.
  Moreover,
  \begin{align}
    P_N(z_k)P_{N-1}^*(z_k)=-2,~~(-1)^{N-k}P_{N-1}^*(z_k)<0,~~(-1)^{N-k}P_N(z_k)>0,
  \end{align}
  and
  \begin{align}\label{zeros-zk}
    (-1)^{N-k}[P_N(z_k)-P_{N-1}^*(z_k)/2]\ge2.
  \end{align}
  Let $y_1<\cdots<y_{N-1}$ be the zeros of $P_{N-1}(x)$ with $y_k\in(x_k,x_{k+1})$. We have
  \begin{align}\label{yk-1}
    P_N(y_k)P_{N-1}^*(y_k)=-2,~~(-1)^{N-k}P_{N-1}^*(y_k)<0,~~(-1)^{N-k}P_N(y_k)>0,
  \end{align}
  and
  \begin{align}\label{zeros-yk}
    (-1)^{N-k}[P_N(y_k)-P_{N-1}^*(y_k)/2]\ge2.
  \end{align}
\end{lem}
\begin{proof}
On account of the recurrence relation \eqref{Pn} and \eqref{Pn*},
we have
\begin{align*}
  P_{n+1}(x)P_n^*(x)-P_n(x)P_{n+1}^*(x)=&P_n(x)P_{n-1}^*(x)-P_{n-1}(x)P_n^*(x)
  \nonumber\\=&\cdots=P_1(x)P_0^*(x)-P_0(x)P_1^*(x)=-2.
\end{align*}
In particular, we obtain
\begin{equation}\label{Wronskian}
  P_N(x)P_{N-1}^*(x)-P_{N-1}(x)P_N^*(x)=-2,
\end{equation}
which implies that $P_{N-1}(x_k)P_N^*(x_k)=2$. On account of \eqref{xk}, we have $(-1)^{N-k}P_N^*(x_k)>0$.
For each $k=1,\ldots,N-1$, since $P_N^*(x)$ has opposite signs at $x_k$ and $x_{k+1}$, there exists a zero of $P_N^*(x)$, denoted by $z_k$, in each interval $(x_k,x_{k+1})$. Moreover, \eqref{Wronskian} implies that $P_N(z_k)P_{N-1}^*(z_k)=-1$. Since $z_k\in(x_k,x_{k+1})$, we obtain $(-1)^{N-k}P_N(z_k)>0$, which implies $(-1)^{N-k}P_{N-1}^*(z_k)<0$ and
$$(-1)^{N-k}\left[P_N(z_k)-\frac{P_{N-1}^*(z_k)}{2}\right]=(-1)^{N-k}\left[P_N(z_k)+\frac{1}{P_N(z_k)}\right]\ge2.$$
Let $y_1<\cdots<y_{N-1}$ be the zeros of $P_{N-1}(x)$. It follows from \eqref{Wronskian} that $P_N(y_k)P_{N-1}^*(y_k)=-1$.
Since $y_k\in(x_k,x_{k+1})$ by Lemma \ref{lem-yk}, we have $(-1)^{N-k}P_N(y_k)>0$, which implies $(-1)^{N-k}P_{N-1}^*(y_k)<0$ and
$$(-1)^{N-k}\left[P_N(y_k)-\frac{P_{N-1}^*(y_k)}{2}\right]=(-1)^{N-k}\left[P_N(y_k)+\frac{1}{P_N(y_k)}\right]\ge2.$$
This completes the proof.
\end{proof}

\section{Turning points}
It has been proven in \cite[Lemma 2]{Geronimo86JAT} that the roots of the polynomial equation $g_N(x)=\pm1$ (i.e., $P_N(x)-P_{N-1}^*(x)/2=\pm2$) are real. We call these real roots the turning points. As we shall see in the next section that these turning points are the endpoints of the subintervals on which the continuous part of the orthogonality measure is supported.
\begin{lem}\label{lem-zeros}
  Assume $f\in C^1(\a,\b)\cup C[\a,\b]$. Given $L\in\R$. If $f(\a)<L$, $f(\b)<L$, and $f(c)\ge L$ for some $c\in(\a,\b)$, then the equation $f(x)=L$ has at least two roots (counting multiplicity) in $(\a,\b)$.
\end{lem}
\begin{proof}
  Let $\xi\in(\a,\b)$ be a maximum point such that $f(\xi)=\max_{x\in[\a,\b]}f(x)$. Clearly, $f(\xi)\ge f(c)\ge L$ and $f'(\xi)=0$.
  If $f(\xi)>L$, then by intermediate value theorem, the equation $f(x)=L$ has at least two roots, one in each of the intervals $(\a,\xi)$ and $(\xi,\b)$. If $f(\xi)=L$, then the equation $f(x)=L$ has at least a double root at $\xi$.
\end{proof}
We denote
\begin{equation}\label{b}
  b=\begin{cases}
    1+a, & a\ge1,\\
    2, & a\le1.
  \end{cases}
\end{equation}

\begin{prop}\label{prop-xi}
  Let $b$ be defined as in \eqref{b}.
  The polynomial equation $g_N^2(x)=1$ has $2N$ real roots (counting multiplicity) in $(-b,b)$. Moreover, $\xi$ is a repeated root if and only if it is a double root and the following identities are satisfied:
  \begin{equation}\label{xi}
    P_{N-1}(\xi)=P_N^*(\xi)=0,~~P_N(\xi)=(-1)^{N-k},~~P_{N-1}^*(\xi)=2(-1)^{N-k-1}.
  \end{equation}
\end{prop}
\begin{proof}
  Note that $x_j$ is a zero of $P_N(x)$ if and only if $2x_j$ is an eigenvalue of the tridiagonal matrix
  \begin{align}\label{T}
    T:=\begin{pmatrix}
    2a&1\\
    1&2a\cos(2\pi/N)&1\\
    &\ddots&\ddots&\ddots\\
    &&1&2a\cos[2(N-1)\pi/N]
    \end{pmatrix},
  \end{align}
  where $T_{jj}=2a\cos[2(j-1)\pi/N]$ for $j=1,\ldots,N$. For any $x>b$, the matrices $2xI+T$ and $2xI-T$ are diagonal dominant and consequently nonsingular. This implies that $|x_j|\le b$ for all $j=1,\ldots,N$.

  Recall from \eqref{gN} that $g_N(x)$
  is a polynomial of degree $N$ with a positive leading coefficient. To prove that all the $2N$ zeros of the equation $g_N^2(x) = 1$ lie inside the interval $(-b,b)$, we shall first show that $g_N(b)>1$; namely, $P_{2N-1}(b)-2P_{N-1}(b)>0$.
  Note that $P_0(b)=1$ and $P_1(b)=2b-2a\ge2$. For convenience, we also set $P_{-1}(b)=0$. As $2b-2a \cos(2k \pi/N) \geq 2b-2a \geq 2$, we have from the recurrence relation \eqref{Pn} that
  \begin{equation} \label{Pn-ineq}
   P_{k+1}(b) \geq 2P_k(b) - P_{k-1}(b).
  \end{equation}
It then follows from induction that  $P_{k+1}(b)-P_k(b)\ge P_k(b)-P_{k-1}(b)\ge1$ and $P_k(b) > 0$ for all $k\ge0$.

Next, note that, when $N=2$, we have $P_2(b)=(2b+2a)(2b-2a)-1\ge7$ and $P_3(b)=P_1(b)P_2(b)-P_1(b)>2P_1(b)$.
  When $N=3$, we get
  $$P_5(b)\ge P_3(b)=(2b+a)P_2(b)-P_1(b)\ge3P_2(b)-P_1(b)>2P_2(b).$$
 When $N\ge4$, for any $a \in \mathbb{R}$, we can select an $m\in\{1,\ldots,N-1\}$ such that $a\cos(2m\pi/N)\le0$. Then, it is readily seen that
  $$P_{m+1}(b)\ge 2bP_m(b)-P_{m-1}(b)\ge 4P_m(b)-P_{m-1}(b)\ge 2P_1(b)+2P_m(b)-P_{m-1}(b),$$
  which implies that $P_N(b)-P_{N-1}(b)\ge P_{m+1}(b)-P_m(b)\ge 2P_1(b)+P_m(b)-P_{m-1}(b)\ge2P_1(b)+1$.
  Now, we define $S_n(x):=P_{n+N}(x)-2P_n(x)$, which satisfies the same recurrence relation for $P_n(x)$, with initial conditions $S_0(x)=P_N(x)-2>0$ and $S_1(x)=P_{N+1}(x)-2P_1(x)=2x P_N(x)-P_{N-1}(x)-2P_1(x)$. It then follows from the above approximation that
  \begin{align*}
  S_1(b)-S_0(b) & = 2b P_N(b)-P_{N-1}(b)-2P_1(b) - (P_N(b)-2) \\
  & \geq P_N(b)-P_{N-1}(b)-2P_1(b) + 2 \geq 3.
  \end{align*}
  As $S_n$ satisfies a similar inequality as \eqref{Pn-ineq}, by induction, we have $S_n(b)-S_{n-1}(b)>0$ and $S_n(b)>0$ for all $n\ge1$. In particular, by setting $n=N-1$, we obtain $P_{2N-1}(b)-2P_{N-1}(b)>0$.
  Thus, for each $N\ge2$, we have proved $g_N(b)>1$. Since the polynomials $(-1)^nP_n(-x)$ satisfy the same recurrence relation for $P_n(x)$ where $a$ is replaced with $-a$; namely, $(-1)^nP_n(-x;a)=P_n(x;-a)$, we obtain by symmetry that $(-1)^ng_N(-b)>1$.

  Let $y_1<\cdots<y_{N-1}$ be the zeros of $P_{N-1}(x)$ with $y_k\in(x_k,x_{k+1})$.
  We also denote $y_0:=-b<x_1$ and $y_N:=b>x_N$ such that $g_N(y_N)>1$ and $(-1)^Ng_N(y_0)>1$.
  For each $k=1,\ldots,N-1$, from \eqref{zeros-yk} we have
  $$(-1)^{N-k}g_N(y_k)\ge1,~~(-1)^{N-k}g_N(y_{k-1})\le-1,~~(-1)^{N-k}g_N(y_{k+1})\le-1.$$
  By Lemma \ref{lem-zeros}, the equation $g_N(x)=(-1)^{N-k}$ has at least two roots (counting multiplicity) in $(y_{k-1},y_{k+1})$.
  Moreover, there is one root in $(y_{N-1},y_N)$ for the equation $g_N(x)=1$ and one root in $(y_0,y_1)$ for the equation $g_N(x)=(-1)^N$.
  Consequently, the equation $g_N^2(x)=1$ has $2N$ roots (counting multiplicity) in the interval $(y_0,y_N)$, while the roots in $(y_0,y_1)$ and $(y_{N-1},y_N)$ are simple.

  If $g_N(x)=(-1)^{N-k}$ has a repeated root $\xi\in(y_{k-1},y_{k+1})$, then its multiplicity is $2$ and $\xi$ is also a maximum point of $(-1)^{N-k}g_N(x)$ in $(y_{k-1},y_{k+1})$; namely, $(-1)^{N-k}g_N(x)<1$ for any $x\in(y_{k-1},\xi)\cup(\xi,y_{k+1})$. In view of \eqref{zeros-zk} and \eqref{zeros-yk} in Lemma \ref{lem-ykzk}, we have $(-1)^{N-k}g_N(y_k)\ge1$ and $(-1)^{N-k}g_N(z_k)\ge1$, where $y_k$ and $z_k$ are the zeros of $P_{N-1}(x)$ and $P_N^*(x)$, respectively, in $(x_k,x_{k+1})\subset(y_{k-1},y_{k+1})$. Therefore, the points $y_k$ and $z_k$ must coincide with $\xi$; namely, $\xi=y_k=z_k$ and $P_{N-1}(\xi)=P_N^*(\xi)=0$. The inequality in \eqref{zeros-yk} now becomes an equality, which implies that $P_N(\xi)=(-1)^{N-k}$. This together with \eqref{yk-1} gives $P_{N-1}^*(\xi)=2(-1)^{N-k-1}$. The proof is completed.
\end{proof}

\begin{thm}\label{thm-xi}
  Let $y_1<\cdots<y_{N-1}$ be the zeros of $P_{N-1}(x)$.
  Denote $y_0=-b$ and $y_N=b$, where $b$ is given by \eqref{b}. Let $\xi_1\le\cdots\le\xi_{2N}$ be the roots (counting multiplicity) of $g_N^2(x)=1$. For each $k=1,\ldots,N$, we have $y_{k-1}\le\xi_{2k-1}<\xi_{2k}\le y_k$ and $(-1)^{N-k}P_{N-1}(x)>0$ for $x\in(\xi_{2k-1},\xi_{2k})$.
\end{thm}
\begin{proof}
  Recall from the proof of Proposition \ref{prop-xi} that the equation $g_N(x)=(-1)^{N-k}$ has exactly two roots (counting multiplicity) in $(y_{k-1},y_{k+1})$, for each $k=1,\ldots,N-1$. Denoting these two roots as $\eta_k^-\le\eta_k^+$, we further have $y_{k-1}<\eta_k^-\le y_k\le \eta_k^+<y_{k+1}$.
  For convenience, we also denote $\eta_N^+=\eta_N^-$ to be the unique root of $g_N(x)=1$ in $(y_{N-1},y_N)$ and $\eta_0^+=\eta_0^-$ the unique root of $g_N(x)=(-1)^N$ in $(y_0,y_1)$.
  Since $g_N(x)$ alternates in sign at $\eta_0^+<\eta_1^+<\cdots<\eta_{N-1}^+<\eta_N^-$, it has exactly one zero, denoted by $\eta_k$, in each of the intervals $(\eta_{k-1}^+,\eta_k^+)$ for $k=1,\ldots,N$. Similarly, we note that $g_N(x)$ alternates in sign at $\eta_0^-<\eta_1^-<\cdots<\eta_N^-$. Thus, we have $\eta_{k-1}^-<\eta_k<\eta_k^-$ for $k=1,\ldots,N$. In particular, we obtain $\eta_{k-1}^+<\eta_k<\eta_k^-$ for $k=1,\ldots,N$. Therefore, the zeros of $g_N^2(x)=1$ are ordered as
  $$\eta_0^+<\eta_1^-\le\eta_1^+<\cdots<\eta_{N-1}^-\le\eta_{N-1}^+<\eta_N^-.$$
  This implies that $\xi_{2k}=\eta_k^-$, $\xi_{2k-1}=\eta_{k-1}^+$, and $y_{k-1}\le\xi_{2k-1}<\xi_{2k}\le y_k$ for each $k=1,\ldots,N$. Moreover, since $(-1)^{N-k}P_{N-1}'(y_k)<0$ (with $k=1,\ldots,N-1$) by \eqref{yk-0}, we obtain $(-1)^{N-k}P_{N-1}(x)>0$ for $x\in(\xi_{2k-1},\xi_{2k})\subset(y_{k-1},y_k)$ with $k=1,\ldots,N$.
  This completes the proof.
\end{proof}

\section{Orthogonality measure} \label{Sec:6}
Based on the recurrence relation \eqref{Pn}, we can use the technique in \cite{Ask:Ism} find the orthogonality measure for $P_n$, which is valid for any $a \in \mathbb{R}$ and $N \in \mathbb{N}$.
Denote $\a_j:=2a\cos(2j\pi/N)=a(q^j+q^{-j})$ with $j=0,1,\ldots$.
We first consider the continued fraction
\begin{align} \label{con-fra}
  \vp(z)=\frac{2}{2z-\a_0-}~\frac{1}{2z-\a_1-}~\frac{1}{2z-\a_2-}\cdots,
\end{align}
which is the same as the Stieltjes transform of the orthogonality measure \cite[Section 2.6]{Ismbook}
\begin{align}
  \vp(z)=\lim_{n\to\infty}\frac{P_n^*(z)}{P_n(z)}=\int_\R\frac{d\mu(x)}{z-x}.
\end{align}
\begin{prop}
  The continued fraction defined in \eqref{con-fra} has an explicit expression
\begin{equation}\label{vp-explicit}
  \vp(z)=\frac{2[P_N(z)-g_N(z)-\sqrt{g_N^2(z)-1}]}{P_{N-1}(z)},
\end{equation}
where
$$\sqrt{g_N^2(z)-1}=2^{N-1}\prod_{j=1}^{2N}(z-\xi_j)^{1/2},$$
and $\xi_1<\xi_2\le\xi_3<\xi_4\le\cdots\le\xi_{2N-1}<\xi_{2N}$ are the roots (counting multiplicity) of $g_N^2(x)=1$.
Let $y_1<\cdots<y_{N-1}$ be the zeros of $P_{N-1}(x)$. We further have
\begin{equation}\label{mk}
  m_k:=\lim_{z\to y_k}[(z-y_k)\vp(z)]=
  \begin{cases}
    0,&~|P_N(y_k)|\ge1,\\
    4\sqrt{|g_N^2(y_k)-1|}/|P_{N-1}'(y_k)|,&~|P_N(y_k)|<1.
  \end{cases}
\end{equation}
In particular, if $y_k=\xi_{2k}=\xi_{2k+1}$ is a double root of $g_N^2(x)=1$, then $P_N(y_k)=g_N(y_k)=(-1)^{N-k}$ and $\vp(z)$ has a removable singularity at $y_k$.
\end{prop}
\begin{proof}
By Theorem \ref{thm-k+jN} and Theorem \ref{thm-k+jN*}, we calculate
$$\vp(z)=\lim_{j\to\infty}\frac{P_{k+jN}^*(z)}{P_{k+jN}(z)}
=\lim_{j\to\infty}\frac{P_{k+N}^*(z)U_{j-1}(g_N)-P_k^*(z)U_{j-2}(g_N)}{P_{k+N}(z)U_{j-1}(g_N)-P_k(z)U_{j-2}(g_N)}.
$$
Since $U_j(g_N)\sim(g_N+\sqrt{g_N^2-1})^{j+1}/(2\sqrt{g_N^2-1})$ as $j\to\infty$, we have
$$\vp(z)=\frac{P_{k+N}^*(z)(g_N(z)+\sqrt{g_N^2(z)-1})-P_k^*(z)}{P_{k+N}(z)(g_N(z)+\sqrt{g_N^2(z)-1})-P_k(z)}.$$
The above formula is valid for any $k\ge0$. In particular, by setting $k=0$, we obtain
\begin{equation} \label{vp-pn*}
\vp(z)=\frac{P_N^*(z)}{P_N(z)-g_N(z)+\sqrt{g_N^2(z)-1}}=\frac{P_N^*(z)[P_N(z)-g_N(z)-\sqrt{g_N^2(z)-1}]}{P_N^2(z)-2P_N(z)g_N(z)+1}.
\end{equation}
Recall from \eqref{gN} and \eqref{Wronskian} that
$$P_N^2(z)-2P_N(z)g_N(z)+1=P_N(z)P_{N-1}^*(z)/2+1=P_{N-1}(z)P_N^*(z)/2.$$
Hence, the continued fraction is simplified as in \eqref{vp-explicit}.  We note that $\varphi(z)$ differs from the $m$-function $m_+(z)$ in \cite[Eq. (10.62)]{Lukbook} by a sign difference, that is, $\varphi(z) = - m_+(z)$.

Note that the leading coefficient of $P_N(z)$ is $2^N$. It then follows from \eqref{gN} that
$$g_N^2(z)-1=2^{2N-2}\prod_{j=1}^{2N}(z-\xi_j).$$
By Theorem \ref{thm-xi}, we obtain $\xi_{2k}\le y_k\le\xi_{2k+1}$, which implies
\begin{align*}
  \lim_{z\to y_k}\sqrt{g_N^2(z)-1}=&2^{N-1}\left[\prod_{j=1}^{2k}(y_k-\xi_j)^{1/2}\right]\left[(-1)^{N-k}\prod_{j=2k+1}^{2N}(y_k-\xi_j)^{1/2}\right]
  \\=&(-1)^{N-k}\sqrt{|g_N^2(y_k)-1|}.
\end{align*}
This gives us
\begin{align*}
  m_k:=&\lim_{z\to y_k}[(z-y_k)\vp(z)]=\frac{2[P_N(y_k)-g_N(y_k)-(-1)^{N-k}\sqrt{|g_N^2(y_k)-1|}]}{P_{N-1}'(y_k)}
  \\=&\frac{2\{\sqrt{|g_N^2(y_k)-1|}-(-1)^{N-k}[P_N(y_k)-g_N(y_k)]\}}{(-1)^{N-k-1}P_{N-1}'(y_k)}.
\end{align*}
Recall that $y_1<\cdots<y_{N-1}$ are simple zeros of $P_{N-1}(x)$ which has a positive leading coefficient. We obtain $(-1)^{N-k-1}P_{N-1}'(y_k)>0$. Moreover, it follows from \eqref{gN} and \eqref{yk-1} that
\begin{align*}
  [P_N(y_k)-g_N(y_k)]^2-g_N^2(y_k)=P_N(y_k)[P_N(y_k)-2g_N(y_k)]=P_N(y_k)P_{N-1}^*(y_k)/2=-1,
\end{align*}
and
\begin{align*}
  P_N(y_k)-g_N(y_k)=&\frac{P_N(y_k)+P_{N-1}^*(y_k)/2}{2}
  \\=&\frac{P_N(y_k)-1/P_N(y_k)}{2}=\frac{|P_N(y_k)|-1/|P_N(y_k)|}{2(-1)^{N-k}}.
\end{align*}
Hence, we have $m_k=0$ if $|P_N(y_k)|\ge1$ and $m_k=4\sqrt{|g_N^2(y_k)-1|}/|P_{N-1}'(y_k)|$ if $|P_N(y_k)|<1$.
This proves \eqref{mk}.
\end{proof}
\begin{thm}\label{thm-w}
  Let $\xi_1<\xi_2\le\xi_3<\xi_4\le\cdots\le\xi_{2N-1}<\xi_{2N}$ be the roots (counting multiplicity) of $g_N^2(x)=1$.
  Let $y_1<\cdots<y_{N-1}$ be the zeros of $P_{N-1}(x)$.
  The polynomials $P_n(x)$ are orthogonal with respect to
  \begin{equation} \label{ortho-weight-pn}
    d\mu(x)=w(x)dx+\sum_{k=1}^{N-1}m_kd\d_{y_k}(x),
  \end{equation}
  where
\begin{align}\label{w}
  w(x)=\frac{2\sqrt{|1-g_N^2(x)|}}{\pi |P_{N-1}(x)|}
\end{align}
is positive and integrable on the intervals $\cup_{k=1}^N(\xi_{2k-1},\xi_{2k})$, $d\d_{y_k}(x)$ is the Dirac delta measure at $y_k$, and $m_k$ is the mass given in \eqref{mk}.
If $\xi_{2k}=\xi_{2k+1}$ for some $k=1,\ldots,N-1$, then $w(x)$ has a removable singularity at $\xi_{2k}$.
Moreover, we have the following identity
\begin{align}\label{int=1}
  \sum_{k=1}^N\int_{\xi_{2k-1}}^{\xi_{2k}}\frac{2\sqrt{|1-g_N^2(x)|}}{\pi|P_{N-1}(x)|}dx=1-\sum_{k=1,~|P_N(y_k)|<1}^N\frac{4\sqrt{|g_N^2(y_k)-1|}}{|P_{N-1}'(y_k)|}.
\end{align}
\end{thm}
\begin{proof}
For $x\in(\xi_{2k-1},\xi_{2k})$, the one-sided limits of $\sqrt{g_N^2(z)-1}$ are given by
\begin{align*}
  (\sqrt{g_N^2(z)-1})_\pm=&\lim_{\ep\to0^\pm}2^{N-1}\prod_{j=1}^{2N}(x+i\ep-\xi_j)^{1/2}
\\=&2^{N-1}\left[\prod_{j=1}^{2k-1}(x-\xi_j)^{1/2}\right]\left[(\pm i)^{2N-2k+1}\prod_{j=2k}^{2N}(\xi_j-x)^{1/2}\right]
\\=&\pm i(-1)^{N-k}\sqrt{1-g_N^2(x)}.
\end{align*}
Moreover, by Theorem \ref{thm-xi}, we have $(-1)^{N-k}P_{N-1}(x)=|P_{N-1}(x)|>0$ for $x\in(\xi_{2k-1},\xi_{2k})$.
Consequently, the continuous part of the orthogonality measure is given by
\begin{align*}
  w(x):=&\frac{\vp_-(x)-\vp_+(x)}{2\pi i}
  \\=&\frac{(\sqrt{g_N^2(z)-1})_+-(\sqrt{g_N^2(z)-1})_-}{\pi iP_{N-1}(x)}
  =\frac{2\sqrt{1-g_N^2(x)}}{\pi |P_{N-1}(x)|},
\end{align*}
for $x\in\cup_{k=1}^N(\xi_{2k-1},\xi_{2k})$. If $P_{N-1}(x)$ has a simple zero at the endpoint of $(\xi_{2k-1},\xi_{2k})$, then $1-g_N^2(x)$ also vanishes. Therefore, $w(x)$ has at least an integrable singularity at that point. In particular, $w(x)$ is positive and integrable on $(\xi_{2k-1},\xi_{2k})$.
If two intervals meet at $\xi_{2k}=\xi_{2k+1}$ for some $k=1,\ldots,N-1$, then by Proposition \ref{prop-xi}, $\xi_{2k}$ is a double root of the equation $g_N^2(x)=1$ and a simple zero of $P_{N-1}(x)$. This implies that $w(x)$ has a removable singularity at $\xi_{2k}$. The identity \eqref{int=1} follows from the fact that the total integral of $d\mu(x)$ equals $1$.
\end{proof}

\begin{rem}
The orthogonality measure \eqref{ortho-weight-pn} can be also found in \cite[Theorem 3]{Geronimus40}, \cite[Theorem 2]{Geronimus57}, \cite[Theorem 10.77]{Lukbook}, and \cite[Theorem 2.14]{VanAsschebook}. One may compare the above theorem with Theorem 10.77 in \cite{Lukbook}, where the discrete mass $m_k$ in \eqref{ortho-weight-pn} is given explicitly in \eqref{mk}.
\end{rem}

In Figure \ref{fig-dmu}, we illustrate the intervals of orthogonality $\cup_{k=1}^N(\xi_{2k-1},\xi_{2k})$ and the mass points $y_k$ that carry a positive mass for $N=1,\ldots,15$. It is noted that a double root of $g_N^2(x)=1$ occurs if and only if $N$ is a multiple of $4$; in this case, the double root is located at $\xi_N=\xi_{N+1}=0$.

\begin{figure}[htp]
\begin{center}
\includegraphics[width=\textwidth]{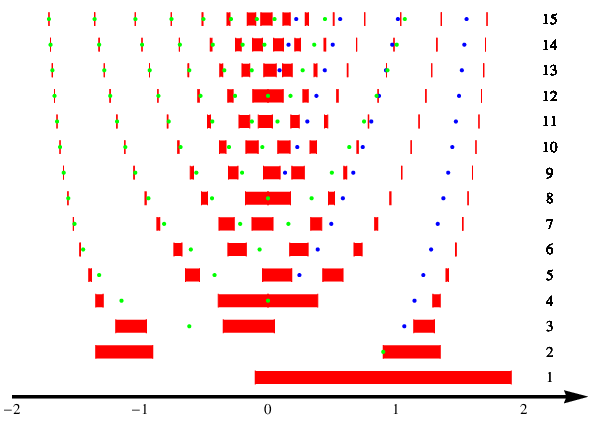}
\end{center}
\caption{The intervals of orthogonality $\cup_{k=1}^N(\xi_{2k-1},\xi_{2k})$ (red bars) and the mass points $y_k$ that carry a positive mass (blue dots) or zero mass (green dots) for $N=1,\ldots,15$. The parameter value is chosen as $a=0.9$.}
\label{fig-dmu}
\end{figure}

Introduce two tri-diagonal matrices, which are submatrices of $T$ in \eqref{T}:
\begin{align}
  T^-=&\begin{pmatrix}
    2a\cos(2\pi/N)&1\\
    1&2a\cos(4\pi/N)&1\\
    &\ddots&\ddots&\ddots\\
    &&1&2a\cos[2(N-1)\pi/N]
  \end{pmatrix},\\
  T^+=&\begin{pmatrix}
    2a&1\\
    1&2a\cos(2\pi/N)&1\\
    &\ddots&\ddots&\ddots\\
    &&1&2a\cos[2(N-2)\pi/N]
  \end{pmatrix},
\end{align}
such that $T^-_{jj}=2a\cos(2j\pi/N)$ and $T^+_{jj}=2a\cos[2(j-1)\pi/N]$ for $j=1,\ldots,N-1$, and $T^\pm_{j,j+1}=T^\pm_{j+1,j}=1$.
The zeros of $P_N^*$ and $P_{N-1}$ correspond to the eigenvalues of $T^-$ and $T^+$, respectively. Hence, in view of Proposition \ref{prop-xi}, the equation $g_N^2(x)=1$ has a double root only if $T^-$ and $T^+$ have a common eigenvalue. We make the following conjecture.
\begin{conj}
  Assume $a>0$. The two tri-diagonal matrices $T^-$ and $T^+$ have a common eigenvalue if and only if $N$ is divisible by $4$; and in this case, the common eigenvalue is $0$.
\end{conj}
For small values of $N$, one can prove the conjecture by direct computations. For instance, when $N=4$, a simple calculation shows that the eigenvalues of $T^-$ are $0$ and $-a\pm\sqrt{a^2+2}$, while the eigenvalues of $T^+$ are $0$ and $\pm\sqrt{4a^2+2}$. Hence, there exists a unique common eigenvalue $0$.

\begin{rem}
  Let $\th=\arccos[g_N(x)]$. We obtain from Theorem \ref{thm-k+jN} and Theorem \ref{thm-w} that
  $$P_{jN-1}(x)w(x)=\frac{2}{\pi}\sin(j\th)=U_{j-1}(g_N)w_U(g_N),~~j\ge1,$$
  where $w(x)$ is given in \eqref{w} and
  $$w_U(g_N)=\frac{2}{\pi}\sqrt{1-g_N^2}=\frac{2}{\pi}\sin\th$$ is the (normalized) orthogonality weight function of Chebyshev polynomials $U_n$.
\end{rem}

\section{Doubly infinite Jacobi matrix} \label{Sec:7}

The three-term recurrence relation \eqref{Pn} is defined for $n \geq 1$. In \cite{Mas:Rep}, Masson and Repka extend their analysis to consider the recurrence relation for $n \in \mathbb{Z}$; see also \cite[Section 7.3]{Berbook}. When we express \eqref{Pn} in matrix form, we obtain the following doubly infinite tridiagonal Jacobi matrix
\begin{equation} \label{2-side-Jacobi}
  A=\begin{pmatrix}
    \ddots&\ddots&\ddots\\
    & \frac{1}{2} & a_{-1} & \frac{1}{2}\\
    && \frac{1}{2} & a_0 & \frac{1}{2}\\
    &&& \frac{1}{2} & a_1 & \frac{1}{2} \\
    &&&&\ddots&\ddots&\ddots\\
  \end{pmatrix},
\end{equation}
with $a_j=a_{-j}$
for $j \in \mathbb{Z}$. The three-term recurrence relation corresponds to the special case $a_j = a \cos(2j\pi/N) = \frac{a}{2}(q^j + q^{-j})$. Here, to align with the notations in \cite{Mas:Rep}, we divide both sides of \eqref{Pn} by 2. It then follows from \cite[Corollary 2.2]{Mas:Rep} that the doubly infinite Jacobi matrix $A$ is self-adjoint. As a consequence, its spectral measure is related to a four-element matrix of measures
\begin{equation}\label{dmu}
d\mu(x)=\left(  \begin{matrix} d\mu_{00}(x) & d\mu_{01}(x)
\\
d\mu_{10}(x) & d\mu_{11}(x) \end{matrix} \right)
\end{equation}
with $d\mu_{01}(x) = d\mu_{10}(x)$. It is noted that both
$d\mu_{00}$ and $d\mu_{11}$ are positive probability measures but
$d\mu_{01}=d\mu_{10}$ is a signed measure.

To compute the measure $d\mu_{ij}(x)$, it is important to observe that its Stieltjes transform corresponds to a matrix element of the resolvent $(zI-A)^{-1}$:
\begin{equation} \label{cauchy-mu-ij}
    S_{ij}(z):=\int_\mathbb{R}{d\mu_{ij}(x)\over z-x}=\langle e_i,(zI-A)^{-1}e_j\rangle \qquad \textrm{for } i,j = 0 ,1.
\end{equation}
Moreover, it follows from \cite[Theorem 2.5]{Mas:Rep} that the matrix elements of the resolvent have the following continued fraction representation
\begin{align} \label{masson-cf}
    \langle e_n,(zI-A)^{-1}e_n\rangle \, ={1\over z-a_n+K_{k=n+1}^\infty\left[{-1/4\over z-a_k}\right]+K_{k=1-n}^{\infty}\left[{-1/4\over z-a_{-k}}\right]}
  \end{align}
and
\begin{align} \label{masson-cf2}
    &\langle e_0,(zI-A)^{-1}e_1\rangle
    \nonumber\\=&{1/2 \over \left(z-a_1+K_{k=2}^\infty \Big[ \frac{-1/4}{z-a_k}\Big] \right) \left( z-a_0+K_{k=1}^{\infty} \Big[\frac{-1/4}{z-a_{-k}} \Big] \right) - 1/4},
\end{align}
where $K_{k=1}^\infty[u_k/v_k]$ is the continued fraction defined as
\begin{equation*}
  K_{k=1}^\infty\left[{u_k\over v_k}\right] = \frac{u_1}{v_1 + }~\frac{u_2}{v_2+}~\frac{u_3}{v_3+}\cdots.
\end{equation*}
The continued fraction in \eqref{con-fra} is generalized as
\begin{align} \label{con-fra-general}
  \vp(z)=\frac{1}{z-a_0-}~\frac{1/4}{z-a_1-}~\frac{1/4}{z-a_2-}\cdots.
\end{align}
Then, we have the following result.

\begin{prop}\label{prop-double-J}
With $\vp(z)$ be defined in \eqref{con-fra-general}, we have
\begin{eqnarray}
\langle e_0,(zI-A)^{-1}e_0\rangle \, &=&  \frac{\vp(z)}{2 - (z-a_0) \vp(z)}, \label{resolvet-00} \\
\langle e_0,(zI-A)^{-1}e_1\rangle \, &=& (-2) \, \frac{1-(z-a_0) \vp(z)}{2-(z-a_0) \vp(z)},\label{resolvet-01} \\
\langle e_1,(zI-A)^{-1}e_1\rangle \, &=& -\frac{4}{\vp(z)} \, \frac{1-(z-a_0) \vp(z)}{2-(z-a_0) \vp(z)}.\label{resolvet-11}
\end{eqnarray}
\end{prop}
\begin{proof}
We first note $a_k = a_{-k}$. Next, form \eqref{con-fra-general}, we have
\begin{equation}
\vp(z) = \frac{1}{z-a_0 + K_{k=1}^\infty\left[{-1/4\over z-a_k}\right]}.
\end{equation}
This gives us
\begin{equation}
K_{k=1}^\infty\left[{-1/4\over z-a_k}\right] = K_{k=1}^\infty\left[{-1/4\over z-a_{-k}}\right]  = \frac{1}{\vp(z) } - z + a_0
\end{equation}
and
\begin{equation}
z-a_1+K_{k=2}^\infty \Big[ \frac{-1/4}{z-a_k}\Big] =  \frac{1}{4} \cdot \frac{\vp(z)}{(z-a)\vp(z)-1}.
\end{equation}
Substituting the above formulas into \eqref{masson-cf} and \eqref{masson-cf2}, we get \eqref{resolvet-00}-\eqref{resolvet-11}.
\end{proof}

\begin{rem}
The above proposition can also be obtained by using methods in spectral theory. For example, this can be proved by combining Lemma 10.40, Theorem 10.76 and Corollary 10.80 in \cite{Lukbook}.
\end{rem}

For the special case $a_j = a \cos(2j\pi/N)$, we will find explicit expressions for $d\mu_{ij}$ in \eqref{dmu}, similar to what have been done in \cite{Dai:Ism:Wang}.
\begin{thm}
  The spectral measure for the doubly infinite tridiagonal Jacobi matrix $A$ in \eqref{2-side-Jacobi} with $a_j = a \cos(2j\pi/N)$ is given by
  \begin{align}
\frac{d\mu_{00}(x)}{dx} =& \frac{|P_N^*(x)|}{2\pi\sqrt{1-g_N^2(x)}}, \label{measure-00} \\
\frac{d\mu_{01}(x)}{dx} =& \frac{(x-a)|P_N^*(x)|}{2\pi\sqrt{1-g_N^2(x)}}, \label{measure-01} \\
\frac{d\mu_{11}(x)}{dx} =& \frac{|P_{N-1}(x)|}{\pi\sqrt{1-g_N^2(x)}}, \label{measure-11}
  \end{align}
  for $x\in\cup_{k=1}^N(\xi_{2k-1},\xi_{2k})$, where $\xi_k$ are defined as in Theorem \ref{thm-w}.
\end{thm}
\begin{proof}
  We first observe that
\begin{align}
  P_N(x)=P_{N+1}^*(x)/2,
\end{align}
because the leading coefficients of $P_N(x)$ and $P_{N+1}^*(x)/2$ are the same as $2^N$ and their zeros correspond to the eigenvalues of the tridiagonal matrices
  \begin{align}
    \begin{pmatrix}
    a&1/2\\
    1/2&a\cos(2\pi/N)&1/2\\
    &\ddots&\ddots&\ddots\\
    &&1/2&a\cos[2(N-1)\pi/N]
    \end{pmatrix},
  \end{align}
  and
  \begin{align}
    \begin{pmatrix}
    a\cos(2\pi/N)&1/2\\
    1/2&a\cos(4\pi/N)&1/2\\
    &\ddots&\ddots&\ddots\\
    &&1/2&a\cos[2N\pi/N]
    \end{pmatrix},
  \end{align}
  respectively.
  Next, using \eqref{vp-pn*}, we have
  \begin{align}
    P_N^*(z)\left[\frac{2}{\vp(z)} -(z-a)\right]
    =2\Big[P_N(z)-g_N(z)+\sqrt{g_N^2(z)-1} \Big]-(z-a)P_N^*(z).
  \end{align}
  As $(z-a)P_N^*(z) = \frac{P_{N+1}^*(z)+P_{N-1}^*(z)}{2}$, it then follows from \eqref{gN} that
  \begin{align*}
    2\Big[P_N(z)-g_N(z) \Big]-(z-a)P_N^*(z)
    =P_N(z)+\frac{P_{N-1}^*(z)}{2}-\frac{P_{N+1}^*(z)+P_{N-1}^*(z)}{2}=0.
  \end{align*}
  Combining the above two formulas, we obtain
  \begin{align}
    \frac{2}{\vp(z)}-(z-a)=\frac{2\sqrt{g_N^2(z)-1}}{P_N^*(z)}.
  \end{align}
  With this identity, it is straightforward to see from \eqref{resolvet-00} and \eqref{resolvet-01} that
  \begin{align}
    \langle e_0,(zI-A)^{-1}e_0\rangle=&\frac{P_N^*(z)}{2\sqrt{g_N^2(z)-1}}, \label{resolvet-00-new}  \\
    \langle e_0,(zI-A)^{-1}e_1\rangle=&\frac{(z-a)P_N^*(z)}{2\sqrt{g_N^2(z)-1}}-1.
  \end{align}
  Moreover, we obtain
  $$
    \langle e_1,(zI-A)^{-1}e_1\rangle=\frac{(z-a)^2P_N^*(z)}{2\sqrt{g_N^2(z)-1}}-\frac{2\sqrt{g_N^2(z)-1}}{P_N^*(z)}
    =\frac{[(z-a)P_N^*(z)]^2-4g_N^2(z)+4}{2P_N^*(z)\sqrt{g_N^2(z)-1}}.
  $$
  On account of
  \begin{align*}
    4[(z-a)P_N^*(z)]^2-16g_N^2(z)=&[2P_N(z)+P_{N-1}^*(z)]^2-[2P_N(z)-P_{N-1}^*(z)]^2
    \\=&8P_N(z)P_{N-1}^*(z)=8P_{N-1}(z)P_N^*(z)-16,
  \end{align*}
  we further have
  \begin{align}
    \langle e_1,(zI-A)^{-1}e_1\rangle=\frac{P_{N-1}(z)}{\sqrt{g_N^2(z)-1}}. \label{resolvet-11-new}
  \end{align}
Recall the following results related to the Cauchy transform: let $L$ be be a certain interval (bounded or unbounded), and $w(x)$ be a H\"older continuous function on $L$. The Cauchy transform of $w$ is given by
\begin{equation}
f(z) = \frac{1}{2 \pi i} \int_L \frac{w(t)}{t-z} dt, \qquad z \in \mathbb{C} \setminus L.
\end{equation}
Then, we have the following Plemelj formula
\begin{equation} \label{Plemellj-formula}
f_\pm(x) = \lim_{y\to 0+} f(x \pm i y) =  \pm \frac{1}{2} w(x) + \frac{i}{2\pi} \, \textrm{P.V.}\int_L \frac{w(t)}{x-t} dt, \qquad x \in L;
\end{equation}
for example, see \cite[Eq. (22.1.2) and Eq. (22.1.5)]{Ismbook}. Using the result above, we finally obtain \eqref{measure-00}-\eqref{measure-11} from \eqref{cauchy-mu-ij} and \eqref{resolvet-00-new}-\eqref{resolvet-11-new}.
\end{proof}

\section{Special cases} \label{Sec:8}
Since $a_{j+N}=a_j$, we have from \eqref{con-fra} that
\begin{align}\label{vp}
  \vp(z)=\frac{2}{2z-\a_0-}~\frac{1}{2z-\a_1-}\cdots\frac{1}{2z-\a_{N-1}-\vp(z)/2},
\end{align}
which implies that $\vp(z)$ satisfies a quadratic equation.

{\bf Case I. $N=1$ and $q=1$.} From \eqref{P-eq}, we have
\begin{align}
  P(t)=\frac{1}{t^2-2(x-a)t+1}=\frac{A}{t-x_+}-\frac{A}{t-x_-}=\sum_{n=0}^\infty (Ax_-^{-n-1}-Ax_+^{-n-1})t^n,
\end{align}
where $x_\pm=(x-a)\pm\sqrt{(x-a)^2-1}$ and
$$A=\frac{1}{x_+-x_-}=\frac{1}{2\sqrt{(x-a)^2-1}}.$$
Since $x_+x_-=1$, we obtain
\begin{equation}
  P_n(x)=\frac{x_+^{n+1}-x_-^{n+1}}{x_+-x_-}.
\end{equation}
The equation for $\vp(z)$ in \eqref{vp} is
\begin{equation}
  \vp(z)=\frac{2}{2z-2a-\frac{\vp(z)}{2}}.
\end{equation}
That is, $\vp^2(z)-4(z-a)\vp(z)+4=0$.
On account of $z\vp(z)\to1$ as $z\to\infty$, we obtain
\begin{equation} \label{vp-case1}
  \vp(z)=2(z-a)-2\sqrt{(z-a)^2-1}.
\end{equation}
It then follows from \eqref{Plemellj-formula} that $d\mu(x)$ is supported on $[a-1,a+1]$ and
\begin{align}
  w(x):=\frac{d\mu(x)}{dx}=\frac{\vp_-(x)-\vp_+(x)}{2\pi i}=\frac{2}{\pi}\sqrt{1-(x-a)^2},
\end{align}
for $x\in(a-1,a+1)$.

From \eqref{vp-case1}, we have
$$\vp(z) = \frac{2}{z-a + \sqrt{(z-a)^2- 1} }.$$
 Substituting it into \eqref{resolvet-00}-\eqref{resolvet-11}, we get
\begin{eqnarray*}
\langle e_0,(zI-A)^{-1}e_0\rangle \, = \langle e_1,(zI-A)^{-1}e_1\rangle \, &=&  \frac{1}{\sqrt{(z-a)^2- 1}}, \\
\langle e_0,(zI-A)^{-1}e_1\rangle \, &=&  \frac{z-a}{\sqrt{(z-a)^2- 1}} - 1.
\end{eqnarray*}
This gives us
\begin{eqnarray}
\frac{d\mu_{00}(x)}{dx} = \frac{d\mu_{11}(x)}{dx} &=& \frac{1}{\pi} \frac{1}{\sqrt{1-(x-a)^2}},   \\
\frac{d\mu_{01}(x)}{dx} &=& \frac{1}{\pi} \frac{x-a}{\sqrt{1-(x-a)^2}}
\end{eqnarray}
for $x \in (a-1, a+1). $

{\bf Case II. $N=2$ and $q=-1$.} From \eqref{P-eq}, we have
\begin{align}
  P(t)=\frac{t^2+2(x-a)t+1}{t^4-2(2x^2-2a^2-1)t^2+1}=\frac{A_1}{t-x_+}+\frac{A_2}{t-x_-}+\frac{A_3}{t+x_+}+\frac{A_4}{t+x_-},
\end{align}
where $x_\pm=\sqrt{x^2-a^2}\pm\sqrt{x^2-a^2-1}$ and
\begin{align*}
  A_1=&\frac{(x_+-x_1)(x_+-x_2)}{2x_+(x_+-x_-)(x_++x_-)}=\frac{\sqrt{x^2-a^2}+(x-a)}{4\sqrt{(x^2-a^2)(x^2-a^2-1)}},\\
  A_2=&\frac{(x_--x_1)(x_--x_2)}{2x_-(x_--x_+)(x_++x_-)}=-\frac{\sqrt{x^2-a^2}+(x-a)}{4\sqrt{(x^2-a^2)(x^2-a^2-1)}}=-A_1,\\
  A_3=&\frac{(x_++x_1)(x_++x_2)}{2x_+(x_--x_+)(x_++x_-)}=\frac{-\sqrt{x^2-a^2}+(x-a)}{4\sqrt{(x^2-a^2)(x^2-a^2-1)}},\\
  A_4=&\frac{(x_-+x_1)(x_-+x_2)}{2x_-(x_+-x_-)(x_++x_-)}=\frac{\sqrt{x^2-a^2}-(x-a)}{4\sqrt{(x^2-a^2)(x^2-a^2-1)}}=-A_3,
\end{align*}
with $x_1=-(x-a)+\sqrt{(x-a)^2-1}$ and $x_2=-(x-a)-\sqrt{(x-a)^2-1}$.
Consequently, we have
\begin{align}
  P_n(x)=-\frac{A_1}{x_+^{n+1}}-\frac{A_2}{x_-^{n+1}}-\frac{A_3}{(-x_+)^{n+1}}-\frac{A_4}{(-x_-)^{n+1}}.
\end{align}
Since $x_+x_-=1$, we obtain
\begin{align}
  P_n(x)=-A_1x_-^{n+1}+A_1x_+^{n+1}-A_3(-x_-)^{n+1}+A_3(-x_+)^{n+1}.
\end{align}
The equation for $\vp(z)$ in \eqref{vp} is
\begin{equation}
  \vp(z)=\frac{2}{2z-2a-\frac{1}{2z+2a-\frac{\vp(z)}{2}}}.
\end{equation}
The continued fraction leads to a quadratic equation
\begin{equation}
  (z-a)\vp(z)^2-4(z^2-a^2)\vp(z)+4(z+a)=0.
\end{equation}
Since $z\vp(z)\to1$ as $z\to\infty$, we choose the root
\begin{equation} \label{vp-case2}
  \vp(z)=\frac{2(z^2-a^2)-2\sqrt{(z^2-a^2)(z^2-a^2-1)}}{z-a}.
\end{equation}
It then follows from \eqref{Plemellj-formula} that $d\mu(x)$ is supported on $[-\sqrt{a^2+1},-a]\cup[a,\sqrt{a^2+1}]$ and
\begin{align}
  w(x):=\frac{d\mu(x)}{dx}=\frac{\vp_-(x)-\vp_+(x)}{2\pi i}=\frac{2}{\pi}\sqrt{\frac{|x+a|}{|x-a|}(a^2+1-x^2)},
\end{align}
for $x\in(-\sqrt{a^2+1},-a)\cup(a,\sqrt{a^2+1})$.

From \eqref{vp-case2}, we have
$$\vp(z) = \frac{2(z+a)}{z^2-a^2 + \sqrt{(z^2-a^2) (z^2-a^2-1)} }.$$
 Substituting it into \eqref{resolvet-00}-\eqref{resolvet-11}, we get
\begin{eqnarray*}
\langle e_0,(zI-A)^{-1}e_0\rangle \, &=&  \frac{z+a}{\sqrt{(z^2-a^2) (z^2-a^2-1)} }, \\
\langle e_0,(zI-A)^{-1}e_1\rangle \, &=&  \frac{z^2-a^2}{\sqrt{(z^2-a^2) (z^2-a^2-1)} } - 1, \\
\langle e_1,(zI-A)^{-1}e_1\rangle \, &=&  \frac{z-a}{\sqrt{(z^2-a^2) (z^2-a^2-1)} }.
\end{eqnarray*}
This gives us
\begin{eqnarray}
\frac{d\mu_{00}(x)}{dx} &=& \frac{1}{\pi} \sqrt{\frac{|x+a| }{|x-a| (a^2+1-x^2)}}, \\
\frac{d\mu_{01}(x)}{dx} &=& \frac{\textrm{sgn}(x)}{\pi} \sqrt{ \frac{a^2-x^2}{a^2+1-x^2}}, \\
\frac{d\mu_{11}(x)}{dx} &=& \frac{1}{\pi} \sqrt{\frac{|x-a|}{|x+a|  (a^2+1-x^2)}}
\end{eqnarray}
for  $x \in (-\sqrt{a^2+1},-a)\cup(a,\sqrt{a^2+1})$.

{\bf Case III. $N=3$ and $q=e^{2\pi i/3}$.} After a tedious calculation, we obtain the continued fraction
\begin{align} \label{vp-case3}
  \vp(z)=\frac{2[4z^3-(3a^2+1)z-a^3+a-\sqrt{(2z+a+1)(2z+a-1)G(z)}]}{4z^2-2az-2a^2-1},
\end{align}
where
\begin{align}\label{G}
  G(z):=[2z^2-(a+1)z-(a^2-a+1)][2z^2-(a-1)z-(a^2+a+1)].
\end{align}
The turning points (i.e., roots of $g_3^2(x)=1$) are ordered as below:
\begin{align*}
  &\xi_1=\frac{a-1-\sqrt{9a^2+6a+9}}{4},~
  \xi_2=\frac{-a-1}{2},~
  \xi_3=\frac{a+1-\sqrt{9a^2-6a+9}}{4},\\
  &\xi_4=\frac{-a+1}{2},~
  \xi_5=\frac{a-1+\sqrt{9a^2+6a+9}}{4},~
  \xi_6=\frac{a+1+\sqrt{9a^2-6a+9}}{4}.
\end{align*}
The zeros of $P_2(x)=4x^2-2ax-2a^2-1$ are
\begin{align}
  y_1=\frac{a-\sqrt{9a^2+4}}{4},~~y_2=\frac{a+\sqrt{9a^2+4}}{4}.
\end{align}
The masses at these two points are $m_1=0$ and $m_2=a/\sqrt{a^2+4/9}$; see \eqref{mk}.
The orthogonality measure is given by
\begin{align}
  d\mu(x)=w(x)dx+\frac{a}{\sqrt{a^2+4/9}}d\d_{y_2}(x),
\end{align}
with
\begin{align}
  w(x)=\frac{2\sqrt{|(2x+a+1)(2x+a-1)G(x)|}}{\pi|4x^2-2ax-2a^2-1|}
\end{align}
for $x\in(\xi_1,\xi_2)\cup(\xi_3,\xi_4)\cup(\xi_5,\xi_6)$. Here, $G(x)$ is the quartic polynomial defined in \eqref{G}. It is also noted that the mass point $y_2\in[\xi_4,\xi_5]$. Since the total integral of $d\mu$ is one, we have the following identity
\begin{align}
  \int_{\xi_1}^{\xi_2}+\int_{\xi_3}^{\xi_4}+\int_{\xi_5}^{\xi_6}\frac{2\sqrt{\prod_{k=1}^6|x-\xi_k|}}{\pi|(x-y_1)(x-y_2)|}dx=\frac{\sqrt{a^2+4/9}-a}{\sqrt{a^2+4/9}}.
\end{align}
In the formula above and in a forthcoming one, we denote for brevity:
\begin{equation*}
\int_{a_1}^{b_1} + \int_{a_2}^{b_2} + \int_{a_3}^{b_3} f(x) dx := \int_{a_1}^{b_1} f(x) dx+\int_{a_2}^{b_2} f(x) dx+ \int_{a_3}^{b_3} f(x) dx.
\end{equation*}

From \eqref{vp-case3}, we have
$$\vp(z) = \frac{2(2z+a+1)(2z+a-1)}{4z^3-(3a^2+1)z-a^3+a + \sqrt{(2z+a+1)(2z+a-1)G(z)} }.$$
Substituting it into \eqref{resolvet-00}-\eqref{resolvet-11}, we get
\begin{eqnarray*}
\langle e_0,(zI-A)^{-1}e_0\rangle \, &=&  \frac{(2z+a+1)(2z+a-1)}{\sqrt{(2z+a+1)(2z+a-1)G(z)} }, \\
\langle e_0,(zI-A)^{-1}e_1\rangle \, &=&  \frac{(z-a)(2z+a+1)(2z+a-1)}{\sqrt{(2z+a+1)(2z+a-1)G(z)} } - 1, \\
\langle e_1,(zI-A)^{-1}e_1\rangle \, &=&  \frac{4z^2-2az-2a^2-1}{\sqrt{(2z+a+1)(2z+a-1)G(z)} }.
\end{eqnarray*}
This gives us
\begin{eqnarray}
\frac{d\mu_{00}(x)}{dx} &=& \frac{1}{\pi} \sqrt{\frac{|(2x+a+1)(2x+a-1)| }{|G(x)|}}, \\
\frac{d\mu_{01}(x)}{dx} &=& \frac{x-a}{\pi} \sqrt{\frac{|(2x+a+1)(2x+a-1)| }{|G(x)|}}, \\
\frac{d\mu_{11}(x)}{dx} &=& \frac{1}{\pi} \frac{|4x^2-2ax-2a^2-1|}{\sqrt{|(2x+a+1)(2x+a-1)G(x)|}}
\end{eqnarray}
for  $x \in (\xi_1,\xi_2)\cup(\xi_3,\xi_4)\cup(\xi_5,\xi_6)$.

{\bf Case IV. $N=4$ and $q=i$.} A tedious calculation gives
\begin{align}
  g_4^2(x)-1=16x^2(x^2-a^2-1)(2x^2-2ax-1)(2x^2+2ax-1),
\end{align}
which has zeros ordered as $\xi_1\le\cdots\le\xi_8$, where
\begin{align*}
  \xi_8=-\xi_1=\sqrt{a^2+1},~\xi_7=-\xi_2=\frac{\sqrt{a^2+2}+a}{2},~\xi_6=-\xi_3=\frac{\sqrt{a^2+2}-a}{2},
\end{align*}
and $\xi_4=\xi_5=0$ is a double zero.
The zeros of $P_3(x)=8x^3-(8a^2+4)x$ are
\begin{align}
  y_1=-\sqrt{a^2+1/2},~~y_2=0,~~y_3=\sqrt{a^2+1/2}.
\end{align}
Note that $y_2=\xi_4=\xi_5=0$ is a double root of the equation $g_4^2(x)=1$.
After removing the singularity at $0$, the continued fraction is
\begin{align} \label{vp-case4}
  \vp(z)=\frac{2[2z^3-(2a^2+1)z+a-\sqrt{(z^2-a^2-1)(2z^2-2az-1)(2z^2+2az-1)}]}{2z^2-(2a^2+1)}.
\end{align}
The masses at $y_1$ and $y_3$ are $m_1=0$ and $m_3=a/\sqrt{a^2+1/2}$. The orthogonality measure is
\begin{align}
  d\mu(x)=w(x)dx+\frac{a}{\sqrt{a^2+1/2}}d\d_{y_3}(x),
\end{align}
where
\begin{align}
  w(x)=\frac{2\sqrt{|(x^2-a^2-1)(2x^2-2ax-1)(2x^2+2ax-1)|}}{\pi|2x^2-2a^2-1|}
\end{align}
for $x\in(\xi_1,\xi_2)\cup(\xi_3,\xi_6)\cup(\xi_7,\xi_8)$.
Since the total integral of $d\mu$ is one and $w(x)=w(-x)$, we have the following identity
\begin{align}
  \int_{0}^{\xi_6}+\int_{\xi_7}^{\xi_8}\frac{4\sqrt{|(x^2-a^2-1)(2x^2-2ax-1)(2x^2+2ax-1)|}}{\pi|2x^2-2a^2-1|}dx=\frac{ \sqrt{a^2+1/2}-a}{\sqrt{a^2+1/2}}.
\end{align}

From \eqref{vp-case4}, we have
$$\vp(z) = \frac{2(2z^2+2az-1)}{2z^3-(2a^2+1)z+a+\sqrt{(z^2-a^2-1)(2z^2-2az-1)(2z^2+2az-1)}}.$$
Substituting it into \eqref{resolvet-00}-\eqref{resolvet-11}, we get
\begin{eqnarray*}
\langle e_0,(zI-A)^{-1}e_0\rangle \, &=&  \frac{2z^2+2az-1}{\sqrt{(z^2-a^2-1)(2z^2-2az-1)(2z^2+2az-1)}}, \\
\langle e_0,(zI-A)^{-1}e_1\rangle \, &=&  \frac{(z-a)(2z^2+2az-1)}{\sqrt{(z^2-a^2-1)(2z^2-2az-1)(2z^2+2az-1)}}- 1, \\
\langle e_1,(zI-A)^{-1}e_1\rangle \, &=&  \frac{2z^2-(2a^2+1)}{\sqrt{(z^2-a^2-1)(2z^2-2az-1)(2z^2+2az-1)}}.
\end{eqnarray*}
This gives us
\begin{eqnarray}
\frac{d\mu_{00}(x)}{dx} &=&  \frac{1}{\pi} \sqrt{\frac{|2x^2-2a^2-1|}{|(x^2-a^2-1)(2x^2-2ax-1)|} }, \\
\frac{d\mu_{01}(x)}{dx} &=& \frac{x-a}{\pi} \sqrt{\frac{|2x^2-2a^2-1|}{|(x^2-a^2-1)(2x^2-2ax-1)|} }, \\
\frac{d\mu_{11}(x)}{dx} &=& \frac{1}{\pi} \frac{|2x^2-2a^2-1|}{\sqrt{|(x^2-a^2-1)(2x^2-2ax-1)(2x^2+2ax-1)|}}
\end{eqnarray}
for  $x\in(\xi_1,\xi_2)\cup(\xi_3,\xi_6)\cup(\xi_7,\xi_8)$.

\section*{Acknowledgments}

We are grateful of Grzegorz \'Swiderski for pointing out the existing results in \cite{Geronimus57,Lukbook,Simonbook,VanAsschebook}. Dan Dai was partially supported by grants
from the Research Grants Council of the Hong Kong Special Administrative Region, China (Project No. CityU 11311622, CityU 11306723 and CityU 11301924).

\noindent D. Dai,
Department of Mathematics, City University of Hong Kong, Tat Chee Avenue, Kowloon, Hong Kong \\
email: dandai@cityu.edu.hk

\medskip

\noindent M.E.H. Ismail,
Department of Mathematics,
University of Central Florida, Orlando, FL 32816, USA \\
  email: mourad.eh.ismail@gmail.com

  \medskip

  \noindent X.-S. Wang,
  Department of Mathematics, University of Louisiana at Lafayette, Lafayette, LA 70503, USA     \\
  email: xswang@louisiana.edu

\end{document}